\pdfoutput=1
\RequirePackage{ifpdf}
\ifpdf % We are running pdfTeX in pdf mode
\documentclass[pdftex]{sigma}
\else
\documentclass{sigma}
\fi

\usepackage{bm,setspace,tikz,xcolor}
\usetikzlibrary{arrows,matrix}
\tikzset{dynkdot/.style={circle,draw,scale=.45}}

\definecolor{darkred}{rgb}{0.7,0,0} % darkred color
\newcommand{\defn}[1]{{\color{darkred}\emph{#1}}} % emphasis of a definition

\newcommand{\BB}{\mathcal{B}}

\newcommand{\cc}{{\bm c}}

\newcommand{\g}{\mathfrak{g}}
\newcommand{\ii}{{\mathbf i}}
\newcommand{\iso}{\approx}
\newcommand{\Kp}{\mathrm{Kp}}
\newcommand{\mm}{{\bm m}}
\newcommand{\MM}{\mathcal{M}}
\newcommand{\mybb}[1]{\mathbf{#1}}
\newcommand{\one}{\boldsymbol{1}}

\newcommand{\QQ}{\mybb{Q}}
\newcommand{\R}{\mathcal{R}}

\newcommand{\sage}{\textsc{SageMath}}

\newcommand{\wt}{{\rm wt}}

\newcommand{\virtual}[1]{\widehat{#1}}
\newcommand{\ZZ}{\mybb{Z}}

\newcommand{\vA}{\virtual{A}}
\newcommand{\va}{\virtual{\alpha}}
\newcommand{\vg}{\virtual{\mathfrak{g}}}
\newcommand{\vL}{\virtual{\Lambda}}
\newcommand{\vY}{\virtual{Y}}

\numberwithin{equation}{section}

\newtheorem{thm}{Theorem}[section]

\newtheorem{Lemma}[thm]{Lemma}

\newtheorem{problem}[thm]{Problem}
 { \theoremstyle{definition}
\newtheorem{dfn}[thm]{Definition}
\newtheorem{ex}[thm]{Example}
\newtheorem{Remark}[thm]{Remark}
}

\usepackage{listings}
\lstdefinelanguage{Sage}[]{Python}
{morekeywords={False,True},sensitive=true}
\begin{document}
\allowdisplaybreaks

\newcommand{\arXivNumber}{1707.09638}

\renewcommand{\PaperNumber}{103}

\FirstPageHeading

\ShortArticleName{Virtual Crystals and Nakajima Monomials}
\ArticleName{Virtual Crystals and Nakajima Monomials}

\Author{Ben SALISBURY~$^\dag$ and Travis SCRIMSHAW~$^\ddag$}
\AuthorNameForHeading{B.~Salisbury and T.~Scrimshaw}

\Address{$^\dag$~Department of Mathematics, Central Michigan University, Mount Pleasant, MI 48859, USA}
\EmailD{\href{mailto:ben.salisbury@cmich.edu}{ben.salisbury@cmich.edu}}
\URLaddressD{\url{http://people.cst.cmich.edu/salis1bt/}}

\Address{$^\ddag$~School of Mathematics and Physics, The University of Queensland,\\
\hphantom{$^\ddag$}~St.~Lucia, QLD 4072, Australia}
\EmailD{\href{tcscrims@gmail.com}{tcscrims@gmail.com}}
\URLaddressD{\url{https://people.smp.uq.edu.au/TravisScrimshaw/}}

\ArticleDates{Received March 28, 2018, in final form September 20, 2018; Published online September 26, 2018}

\Abstract{An explicit description of the virtualization map for the (modified) Nakajima monomial model for crystals is given. We give an explicit description of the Lusztig data for modified Nakajima monomials in type $A_n$.}


\Classification{05E10; 17B37}

\section{Introduction}

Diagram folding is an effective technique used to study a non-simply-laced Kac--Moody algebra~$\g$ by embedding $\g$ inside of a simply-laced Kac--Moody algebra $\vg$.  On the level of crystals, Kashiwara~\cite{K96} introduced a method for embedding highest weight $U_q(\g)$-crystals inside highest weight $U_q(\vg)$-crystals, provided the Dynkin diagram of~$\g$ can be obtained from the Dynkin diagram of~$\vg$ using a diagram folding. The embedding is called a virtualization map and resulting image with an induced crystal structure is called a virtual crystal.

It is known that simply-laced crystals have a nice local characterization~\cite{Stembridge03}, a description which is lacking in the non-simply-laced types (see~\cite{DKK09,Sternberg07} for partial results in this direction).  One can develop the theory of crystals using a combination of the Stembridge axioms and the theory of virtual crystals, which is the approach taken in the recent text of Bump and Schilling~\cite{BS17}. Therefore, developing an explicit virtualization map on different models for highest weight crystals and $B(\infty)$, the crystal of the lower half of the quantum group $U_q^-(\g)$, for the correct model could result in an extension of the Stembridge axioms to all symmetrizable types.

In terms of Kashiwara--Nakashima tableaux~\cite{KN94}, which only is valid for classical types,\footnote{Type $G_2$ was given by Kang--Misra tableaux~\cite{KM94}.} a~virtualization map has been explicitly described in~\cite{baker2000, SchillingS15}. However, it is a little complicated with specific rules depending on the type.  To give a more uniform description of a virtualization map, we can use the polyhedral model of~\cite{Nakashima99, NZ97}, which is described by iterating the Kashiwara embedding. This was the model used by Kashiwara in his construction~\cite{K96}, where he also gave precise conditions on when a virtualization map exists from type $\g$ into type $\vg$. This virtuali\-zation is natural in the sense that is commutes with the Kashiwara embedding (in a~suitable sense) and can be computed directly from the diagram folding (and using that virtual crystals form a tensor category~\cite{OSS03III}).

The virtualization map has also been explicitly described for a number of other (uniform) models for highest weight crystals. In~\cite{PS15}, the virtualization map for the Littelmann path model~\cite{L95, L95-2} was completely described, which naturally extends to a virtualization map of the Littelmann path model for $B(\infty)$ of~\cite{LZ11}. This was also examined by Naito and Sagaki in~\cite{NS01}, but the different type of diagram automorphism allowed can have an edge to be fixed under the automorphism but not the corresponding vertices. This was used to construct virtual crystals for certain Kirillov--Reshetikhin crystals~\cite{NS05II,NS10}, a class of finite $U_q'(\g)$-crystals for $\g$ of affine type.

If $\g$ is of finite type, then there is an analog of a PBW basis of the universal enveloping algebra for the quantum group (see, e.g.,~\cite{L90}). Lusztig constructed a canonical basis for $U_q^-(\g)$~\cite{L90}, where it was shown to be equivalent to Kashiwara's crystal basis \cite{GL93}. Moreover, the crystal basis is unitriangular with respect to the PBW basis and precisely one monomial survives in the $q \to 0$ limit. The exponents of this monomial are called the Lusztig datum of the element. The transition functions of~\cite{BZ01} between these PBW bases can be used to describe the crystal structure, but one needs to apply numerous piecewise-linear equations in order to do so. Recently, this was simplified for certain reduced expressions, where a bracketing rule was given on Lusztig data~\cite{SST17,SST16}.

Mirkovi\'c--Vilonen (MV) polytopes are the image of the moment map of MV cycles and are a~model for $B(\infty)$ for finite types~\cite{Kamnitzer07, Kamnitzer10}. Furthermore, MV polytopes encode all of the Lusztig data of an element, and so it is an equivalent model to the PBW bases. Furthermore, the virtualization map on MV polytopes was described explicitly in~\cite{JS15,NS08III}, and hence, extends to the Lusztig data/PBW bases. Note that the virtualization map is natural as it is induced from the corresponding embeddings of word, which comes from the diagram folding. Moreover, for certain choices of reduced words of the long element of the corresponding Weyl group, this naturally extends to the bracketing rules on Lusztig data.

Rigged configurations are a combinatorial model for highest weight crystals and $B(\infty)$ stemming from statistical mechanics~\cite{SalisburyS15,SalisburyS16,S06,SchillingS15}. The virtualization map on rigged configurations is also natural and explicitly determined~\cite{SalisburyS15,SchillingS15}. It was recently shown that the $*$-involution has a natural interpretation on rigged configurations~\cite{SalisburyS16II}, and it is straightforward to see that the $*$-involution commutes with the virtualization map.

The main model we focus on is given by Nakajima monomials, which were originally used to describe the $t$-analog of $q$-characters~\cite{Nakajima01,Nakajima03II,Nakajima03,Nakajima04}. We will primarily be using the crystal structure given by Kashiwara~\cite{K03II}, which is distinct from that given by Nakajima~\cite{Nakajima03}. This can be considered a generalization of the polyhedral model, which follows from looking at the tensor product rule~\cite{K91} and using the $A_{i,k}$ variables (corresponding to simple roots). Therefore, our explicit virtualization map can be considered a generalization of that of~\cite[Theorem~5.1]{K96}, but it also reflects the naturality of the map as the virtualization map can also be described using the $Y_{i,k}$ variables (corresponding to fundamental weights). Our results allow us to explicitly determine the conditions necessary to have a virtualization map, giving another proof of the results in~\cite{K96,PS15}.

We note that the naturality of the virtualization is a reflection of the close connection of Nakajima monomials to geometry; in particular, quiver varieties~\cite{Nakajima04,ST14}. This gives a further relation with the results of~\cite{Sav05}. Our results are further evidence that there should be a simple, explicit method to construct a bijection between Nakajima monomials and other models such as Littelmann paths or rigged configurations. Finding such a bijection between rigged configurations and Nakajima monomials would give a connection between the $*$-involution, the crystal commutor, and quiver varieties~\cite{HK06,KT09,Sav09}.

Let $\g$ be of affine type. We note that our virtualization map extends to the crystal structure on Nakajima monomials as given by Nakajima~\cite{Nakajima03}. As an application of our embedding, we obtain a virtualization map for KR crystals $B^{r,1}$, for certain $r \in I_0$ that also depends on the type, given by Nakajima monomials described in~\cite{HN06}. These are analogous results to~\cite{PS15}, along with some of the cases of~\cite{NS01,NS05II,NS10}. Thus, our results give an alternative proof of certain special cases of~\cite[Conjecture~3.7]{OSS03III}, which was recently settled for non-exceptional types in~\cite{Okado13} using type-specific arguments.

For type $A_n$, we give an explicit map between the modified Nakajima monomial model and the Kostant partition crystal (or equivalently, the PBW crystal) of~\cite{SST16} for a specific reduced word of the long element $w_0$ (the dual BZL word). Thus, we provide an explicit way to obtain the Lusztig data for the specific word from the Nakajima monomial model. A goal of this paper was to obtain similar data for types $B_n$ and $C_n$ using our virtualization map on Nakajima monomials. However, we have a certain aligned condition, which we can consider as compatible orientations on the corresponding Dynkin diagrams, that is not compatible with the reduced expression for~$w_0$. We can combine this with the transition maps of~\cite{BZ01}, but this is not combinatorial or simple to describe. So our approach will need to be modified in order to obtain the Lusztig data from Nakajima monomials in types $B_n$ and $C_n$.

This paper is organized as follows. In Section~\ref{sec:monomials}, we recall the definition of (modified) Nakajima monomials. In Section~\ref{sec:virtual}, we recall the definition of virtual crystals. In Section~\ref{sec:virtualization_map}, we construct the virtualization map on Nakajima monomials. In Section~\ref{sec:KR_crystals}, we extend our virtualization map to certain KR crystals as given by~\cite{HN06}. In Section~\ref{sec:monomials_pbw}, we give an explicit map between Nakajima monomials and PBW crystals or Lusztig data. In Section~\ref{sec:problems}, we give some open problems from our work.

\section{Crystals and Nakajima monomials}\label{sec:monomials}

Since we will only be working with one particular model for crystals, we explicitly describe the properties of that model and leave the generalities for the interested reader to pursue.  (See, for example,~\cite{BS17,HK02}.)

Let $C = (C_{ij})_{i,j\in I}$ be a generalized Cartan matrix for a symmetrizable Kac--Moody algebra $\g$ with index set $I$, weight lattice $P$, fundamental weights $\{\Lambda_i\colon i\in I\}$, and simple roots $\{\alpha_i \colon i\in I\}$.  Let $P^{\vee}$ denote the coroot lattice.  For $\lambda \in P$ and $h \in P^{\vee}$, let $\langle h, \lambda \rangle = \lambda(h)$ be the canonical pairing.  Recall that $\Lambda_j(h_i) = \delta_{ij}$ (Kronecker delta) and $\alpha_j(h_i) = C_{ij}$, where $\{ h_i\colon i\in I \}$ is the set of simple coroots in $P^\vee$. We identity simple roots in the weight space $\QQ\otimes_\ZZ P$ via the Cartan matrix; that is, in $\QQ\otimes_\ZZ P$, we have $\alpha_i = \sum\limits_{j\in I} C_{ji}\Lambda_i$. We write $i \sim j$ if $i \neq j$ and $C_{ij} \neq 0$; i.e., if $i$ and $j$ are adjacent in the Dynkin diagram of~$\g$.

Let $\MM = \{Y_{i,k}\colon i \in I ,\, k \in \ZZ \} \cup \{\one\}$ be a set of commuting variables and in which $\one$ denotes the multiplicative identity. Choose an array $\cc = (c_{ij}\colon i,j\in I,\, i\neq j)$ of integers which satisfy the condition $c_{ij} + c_{ji} = 1$. Define $\MM_{\cc}$ as $\MM$ with a $U_q(\g)$-crystal structure as follows. For a~monomial $M = \prod\limits_{i\in I} \prod\limits_{k\in \ZZ} Y_{i,k}^{y_i(k)}$, define
\begin{subequations}\label{eq:mon_cry_ops}
\begin{gather}
\mathrm{wt}(M)  = \sum_{i\in I} \bigg( \sum_{k \in \ZZ} y_i(k) \bigg) \Lambda_i, \\
\varphi_i(M)  = \max\bigg\{ \sum_{j \le k} y_i(j)\colon k \in \ZZ \bigg\}, \\
\varepsilon_i(M)  = \varphi_i(M) - \langle h_i, \mathrm{wt}(M) \rangle.\label{eq:oldep}
\\ f_i M   = \begin{cases}
M A^{-1}_{i,k_i^f} & \text{if } \varphi_i(M) > 0,\\
0 & \text{if } \varphi_i(M) = 0,
\end{cases}
\\ e_i M   = \begin{cases}
M A_{i,k_i^e} & \text{if } \varepsilon_i(M) > 0,\\
0 & \text{if } \varepsilon_i(M) = 0,
\end{cases}
\end{gather}
\end{subequations}
where
\begin{gather*}
A_{i,k}  = Y_{i,k}Y_{i,k+1} \prod_{j \neq i}Y_{j,k+c_{j,i}}^{C_{ji}}, \\
k_i^f = k_i^f(M)  = \min\bigg\{ k\colon \varphi_i(M) = \sum_{j \le k} y_i(j) \bigg\}, \\
k_i^e = k_i^e(M)  = \max\bigg\{ k\colon \varphi_i(M) = \sum_{j \le k} y_i(j) \bigg\}.
\end{gather*}

\begin{Remark}We can also define $\varepsilon_i$ by
\begin{gather*}
\varepsilon_i(M) = \max\bigg\{ {-} \sum_{j>k} y_i(j)\colon k \in \ZZ \bigg\}.
\end{gather*}
\end{Remark}

Let $\MM(M)_{\cc}$ denote the closure of a monomial $M$ under the crystal operators given above, and let $B(\lambda)$ denote the $U_q(\g)$-crystal associated to the irreducible highest weight representa\-tion~$V(\lambda)$ of $U_q(\g)$ with highest weight $\lambda \in P$.

\begin{thm}[Kashiwara~\cite{K03II}]\label{thm:highest_weight_model} Let $M$ be a monomial such that $e_i M = 0$ for all $i \in I$. Then $\MM(M)_{\cc} \cong B\bigl(\wt(M)\bigr)$.
\end{thm}

In particular, denote $\MM(\lambda)_{\cc} := \MM(Y_{\lambda})_{\cc}$, where
\begin{gather*}
Y_{\lambda} = \prod_{i \in I} Y_{i,0}^{\langle h_i, \lambda \rangle}.
\end{gather*}
This is referred to as the $B(\lambda)$ model using \defn{Nakajima monomials}.

Next, suppose $c_{ij} \in \ZZ_{\ge0}$ for all $i$, $j$. We define $\MM(\infty)_{\cc}$ as the closure of $\one$ under the modified Kashiwara operators
\begin{gather*}
\overline{f}_i M = MA_{i,\overline{k}_i^f}^{-1},
\end{gather*}
where
\begin{gather*}
\overline{k}_i^f = \min\bigg\{ k\ge 0\colon \varphi_i(M) = \sum_{0\le j\le k} y_i(j) \bigg\},
\end{gather*}
with the remaining crystal structure being the same as for $\MM_{\cc}$.  Note if $M \in \MM(\infty)_\cc$, then $M$ has the form
\begin{gather}\label{eq:Minfge0}
M = \prod_{i\in I} \prod_{k\ge 0} Y_{i,k}^{y_i(k)}.
\end{gather}
When there is no danger of confusion, we will simply write the modified Kashiwara operator as~$f_i$. The crystal $\MM(\infty)_{\cc}$ is called the crystal of \defn{modified Nakajima monomials}.

\begin{thm}[Kang--Kim--Shin \cite{KKS07}] Let $\g$ be of symmetrizable type. Then $\MM(\infty)_\cc \cong B(\infty)$.
\end{thm}

Let $\lambda \in P$ be a dominant integral weight.  Define $R_\lambda = \{r_\lambda\}$ to be the one-element abstract $U_q(\g)$-crystal whose operations, for all $i \in I$, are defined as
\begin{gather*}
e_i r_\lambda = f_i r_\lambda = 0,
\qquad
\varepsilon_i(r_{\lambda}) = -\langle h_i, \lambda \rangle,
\qquad
\varphi_i(r_\lambda) = 0,
\qquad
\wt(r_\lambda) = \lambda.
\end{gather*}
By \cite{K93}, there is a strict crystal embedding $B(\lambda) \lhook\joinrel\longrightarrow B(\infty) \otimes R_\lambda$, where $\otimes$ denotes the crystal tensor product defined in~\cite{K91}.  We will not require the general definition of tensor product of crystals, but note that the connected component containing $u_{\infty} \otimes r_{\lambda}$, where $u_{\infty}$ is the highest weight element of $B(\infty)$, is isomorphic to $B(\lambda)$. The crystal operators of the tensor product $M \otimes r_{\lambda} \in \MM(\infty)_{\cc} \otimes R_{\lambda}$ precisely corresponds to taking $Y_{\lambda} M$ for any $M \in \MM(\infty)_{\cc}$ and using the (unmodified) crystal structure of~\eqref{eq:mon_cry_ops}. Going in the opposite direction, we can construct~$\MM(\infty)_{\cc}$ as the direct limit
\begin{gather*}
\MM(\infty)_{\cc} = \varinjlim_{\lambda \to \infty} Y_{\lambda}^{-1} \MM(Y_{\lambda}),
\end{gather*}
where we consider the partial order $\lambda \leq \mu$ if and only if $\langle h_i, \lambda \rangle \leq \langle h_i, \mu \rangle$ for all $i \in I$. We refer the reader to~\cite{K02} for a precise definition.

\section{Virtual crystals} \label{sec:virtual}

A \defn{diagram folding} is a surjective map $\phi\colon \virtual{I} \longrightarrow I$ between index sets of Kac--Moody algebras and a set $(\gamma_i \in \ZZ_{>0}\colon i \in I)$ of \defn{scaling factors}.  One may induce a~map from $\phi$ on the corresponding weight lattices $\widetilde{\phi} \colon P \longrightarrow \virtual{P}$ by asserting
\begin{gather}
\label{eq:weight_embedding}
\Lambda_i \mapsto \gamma_i \sum_{i' \in \phi^{-1}(i)} \vL_{i'}.
\end{gather}
If there exists an embedding $\iota\colon \g \lhook\joinrel\longrightarrow \vg$ of symmetrizable Kac--Moody algebras that induces $\widetilde{\phi}$, then $\iota$ induces an injection $v \colon B(\lambda) \lhook\joinrel\longrightarrow B(\virtual{\lambda})$ as sets, where $\virtual{\lambda} := \widetilde{\phi}(\lambda)$. Suppose the crystal structure on $B(\lambda)$ is denoted by $(e_i,f_i,\varepsilon_i,\varphi_i,\wt)$ and the crystal structure on $B(\virtual{\lambda})$ is denoted by $(\virtual{e}_i, \virtual{f}_i, \virtual{\varepsilon}_i, \virtual{\varphi}_i, \virtual{\wt})$. Then there is additional structure on the image under $v$ as a \defn{virtual crystal}, where the action $e_i$ and $f_i$ are defined on the image as
\begin{gather}
\label{eq:virtual_crystal_ops}
e^v_i = \prod_{i' \in \phi^{-1}(i)} \virtual{e}_{i'}^{\,\gamma_i}
\quad \quad
\text{ and }
\quad \quad
f^v_i = \prod_{i' \in \phi^{-1}(i)} \virtual{f}_{i'}^{\,\gamma_i},
\end{gather}
respectively, and commute with $v$~\cite{baker2000,OSS03III,OSS03II}; i.e., $e_i^v\circ v = v\circ e_i$ for all $i\in I$, etc. These are known as the \defn{virtual Kashiwara (crystal) operators}. It is shown in~\cite{K96} that for any $i \in I$ and $i_1',i_2' \in \phi^{-1}(i)$ we have $\virtual{e}_{i_1'} \virtual{e}_{i_2'} = \virtual{e}_{i_2'} \virtual{e}_{i_1'}$ and $\virtual{f}_{i_1'} \virtual{f}_{i_2'} = \virtual{f}_{i_2'} \virtual{f}_{i_1'}$ as operators (recall that the nodes indexed by $i_1'$ and $i_2'$ in the Dynkin diagram of $\vg$ are not connected), so both $e^v_i$ and $f^v_i$ are well-defined. The inclusion map $v$ also satisfies the relation $\widetilde{\phi} \circ \wt = \virtual{\wt}\circ v$.
In~\cite{baker2000}, it was shown that this defines a $U_q(\g)$-crystal structure on the image of $v$.  More generally, we define a virtual crystal as follows.

\begin{dfn}\label{def:virtual}Consider any symmetrizable types $\g$ and $\vg$ with index sets $I$ and $\virtual{I}$, respectively. Let $\phi \colon \virtual{I} \longrightarrow I$ be a surjection such that $\virtual{C}_{i_1'i_2'} = 0$ for all $i_1',i_2' \in \phi^{-1}(i)$ and $i \in I$. Let $\virtual{B}$ be a~$U_q(\vg)$-crystal and $V \subseteq \virtual{B}$.  Let $\gamma = (\gamma_i \in \ZZ_{>0} \colon i \in I)$ be the scaling factors. A \defn{virtual crystal} is the quadruple $(V, \virtual{B}, \phi, \gamma)$ such that $V$ has an abstract $U_q(\g)$-crystal structure defined using the Kashiwara operators $e_i^v$ and $f_i^v$ from~\eqref{eq:virtual_crystal_ops} above,
\begin{gather*}
\varepsilon_i^v(x)  := \frac{\virtual{\varepsilon}_{i'}(x)}{\gamma_i}, \qquad
\varphi_i^v(x)  := \frac{\virtual{\varphi}_{i'}(x)}{\gamma_i},  \qquad \text{for all}  \quad i'\in \phi^{-1}(i) \quad \text{and}\quad x \in V,
\end{gather*}
and $\wt^v := \widetilde{\phi}^{-1} \circ \virtual{\wt}$.
\end{dfn}
We say $B$ \defn{virtualizes} in $\virtual{B}$ if there exists a $U_q(\g)$-crystal isomorphism $v \colon B \longrightarrow V$.  The resulting isomorphism is called the \defn{virtualization map}. We denote the quadruple $(V,\virtual{B},\phi,\gamma)$ simply by $V$ when there is no risk of confusion.

\section{The virtualization map} \label{sec:virtualization_map}

Assume the notations of the previous sections.  We will say $\virtual{\cc} = (\virtual{c}_{i'j'})_{i'\neq j'}$ is \defn{compatible} with $\cc$ when $\virtual{c}_{i'j'} = c_{ij}$ for all $i' \in \phi^{-1}(i)$ and $j' \in \phi^{-1}(j)$ satisfying $i' \sim j'$ and $i \sim j$. For this section assume that $\virtual{\cc}$ is compatible with $\cc$. Let $\virtual{\MM}_{\virtual{\cc}}$ be the set of modified Nakajima monomials of type $\vg$ in the variables $\{\vY_{i,k}\colon i \in \virtual{I},\, k\in \ZZ\}$ and identity $\virtual{\one}$ equipped with $U_q(\vg)$-crystal structure denoted by $\virtual{e}_i$, $\virtual{f}_i$, $\virtual{\varepsilon}_i$, $\virtual{\varphi}_i$, and $\virtual{\wt}$. With these definitions, define a map
\begin{gather*}
v\colon \ \MM_\cc \longrightarrow \virtual{\MM}_{\virtual{\cc}} \qquad \text{by}\qquad  Y_{i,k} \mapsto \prod_{i' \in \phi^{-1}(i)} \vY_{i',k}^{\gamma_i}
\end{gather*}
and extend multiplicatively. By abuse of notation, we also consider $v \colon \MM(\infty)_{\cc} \longrightarrow \virtual{\MM}(\infty)_{\virtual{\cc}}$.

\begin{Lemma}\label{lem:AtoAhat} Suppose
\begin{gather*}
\widetilde{\phi}(\alpha_i) = \displaystyle\sum_{i'\in\phi^{-1}(i)} \gamma_i\va_{i'}
\end{gather*}
for all $i \in I$.  Then, for all $i\in I$ and $k \in \ZZ_{\ge0}$,
\begin{gather*}
v(A_{i,k}) = \prod_{i'\in \phi^{-1}(i)} \vA_{i',k}^{\gamma_i}.
\end{gather*}
\end{Lemma}

\begin{ex}Let
\begin{gather*}
C = \begin{pmatrix} 2 & -2 \\ -1 & 2 \end{pmatrix}, \!\!\!\qquad  \cc = \begin{pmatrix}\bullet &0\\1&\bullet\end{pmatrix},\!\!\! \qquad\text{and}\!\!\! \qquad  \virtual{C} = \begin{pmatrix} 2 & -1 & 0 \\ -1 & 2 & -1 \\ 0 & -1 & 2 \end{pmatrix}, \!\!\!\qquad \virtual{\cc} = \begin{pmatrix}\bullet & 0 & 0 \\ 1 & \bullet & 1 \\ 1 & 0 &  \bullet \end{pmatrix}.
\end{gather*}
Here, $C$ is a Cartan matrix of type $C_2$ and $\virtual{C}$ is a Cartan matrix of type $A_3$.  For arbitrary scaling factors $(\gamma_i \colon i\in I)$, the condition $\widetilde{\phi}(\alpha_i) = \sum\limits_{j\in \phi^{-1}(i)} \gamma_i\va_j$ imposes a linear dependence relation amongst the $\gamma_i$.  In this example, we have
\begin{gather*}
\widetilde{\phi}(\alpha_1)  = \widetilde{\phi}(2\Lambda_1 - \Lambda_2) = 2\widetilde{\phi}(\Lambda_1)- \widetilde{\phi}(\Lambda_2) = 2\gamma_1(\vL_1 + \vL_3) - \gamma_2\vL_2,\\
\gamma_1(\va_1+\va_3)  = \gamma_1(2\vL_1 - \vL_2 - \vL_2 + 2\vL_3) = 2\gamma_1(\vL_1 - \vL_2 + \vL_3),
\end{gather*}
so the relation $2\gamma_1 = \gamma_2$ is required to ensure $\widetilde{\phi}(\alpha_1) = \gamma_1(\va_1+\va_3)$.  Similarly, for $k\ge 0$,
\begin{gather*}
v(A_{1,k})  = v(Y_{1,k}) v(Y_{1,k+1}) \prod_{j\neq i} v(Y_{j,k+c_{j1}})^{C_{j1}} = \vY_{1,k}^{\gamma_1} \vY_{3,k}^{\gamma_1} \vY_{1,k+1}^{\gamma_1} \vY_{3,k+1}^{\gamma_1} \vY_{2,k+1}^{-\gamma_2}, \\
\prod_{i'\in\phi^{-1}(1)} \vA_{i',k}^{\gamma_i}  = \vA_{1,k}^{\gamma_1} \vA_{3,k}^{\gamma_1} = \vY_{1,k}^{\gamma_1} \vY_{1,k+1}^{\gamma_1} \vY_{2,k+1}^{-\gamma_1} \vY_{3,k}^{\gamma_1} \vY_{3,k+1}^{\gamma_1} \vY_{2,k+1}^{-\gamma_1}.
\end{gather*}
These two are equal if and only if $2\gamma_1 = \gamma_2$.
\end{ex}

\begin{proof}First, let's examine the condition that $\widetilde{\phi}(\alpha_i) = \sum\limits_{i'\in\phi^{-1}(i)}\gamma_i\va_{i'}$.  Recall that the Cartan matrix is the change of basis matrix from the basis of simple roots to the basis of fundamental weights.  So, on the one hand, for $i\in I$, we have
\begin{gather*}
\widetilde{\phi}(\alpha_i) = \widetilde{\phi}\bigg( \sum_{j\in I} C_{ji}\Lambda_j \bigg) = \sum_{j\in I} C_{ji} \widetilde{\phi}(\Lambda_j) = \sum_{j\in I} C_{ji} \gamma_j \sum_{j'\in\phi^{-1}(j)} \vL_{j'},
\end{gather*}
while, on the other hand, we have
\begin{gather*}
\widetilde\phi(\alpha_i) = \sum_{i'\in\phi^{-1}(i)} \gamma_i\va_{i'} = \sum_{i' \in \phi^{-1}(i)}\gamma_i \sum_{j'\in \virtual{I}} \virtual{C}_{j'i'}\vL_{j'} = \sum_{j'\in \virtual{I}} \sum_{i'\in\phi^{-1}(i)} \gamma_i \virtual{C}_{j'i'}\vL_{j'}.
\end{gather*}
Comparing the coefficient of $\vL_{j'}$ from both of these equations, we conclude, for all $i\in I$,
\begin{gather}\label{eq:lindep}
\gamma_jC_{ji} = \sum_{i'\in\phi^{-1}(i)} \gamma_i \virtual{C}_{j'i'} \qquad \text{for any $j$ such that} \quad j' \in \phi^{-1}(j).
\end{gather}

Let $i\in I$ and $k\ge 0$ be fixed.  Then
\begin{gather*}
v(A_{i,k}) = \prod_{i'\in\phi^{-1}(i)} \vY_{i',k}^{\gamma_i} \vY_{i',k+1}^{\gamma_i} \prod_{j\neq i}\prod_{j' \in \phi^{-1}(j)} \vY_{j',k+c_{ji}}^{\gamma_jC_{ji}}.
\end{gather*}
Since $\virtual{c}_{i'j'} = c_{ij}$ for all $i' \in \phi^{-1}(i)$ and $j' \in \phi^{-1}(j)$ such that $i\sim j$ and $i' \sim j'$, we have
\begin{gather*}
\prod_{i'\in\phi^{-1}(i)} \vY_{i',k}^{\gamma_i} \vY_{i',k+1}^{\gamma_i} \prod_{j\neq i}\prod_{j' \in \phi^{-1}(j)} \vY_{j',k+c_{ji}}^{\gamma_jC_{ji}} = \prod_{i'\in\phi^{-1}(i)} \vY_{i',k}^{\gamma_i} \vY_{i',k+1}^{\gamma_i} \prod_{j\neq i}\prod_{j' \in \phi^{-1}(j)} \vY_{j',k+\virtual{c}_{j'i'}}^{\gamma_jC_{ji}}.
\end{gather*}
By the calculation above,
\begin{gather*}
\prod_{i'\in\phi^{-1}(i)} \vY_{i',k}^{\gamma_i} \vY_{i',k+1}^{\gamma_i} \prod_{j\neq i}\prod_{j' \in \phi^{-1}(j)} \vY_{j',k+\virtual{c}_{j'i'}}^{\gamma_jC_{ji}} = \prod_{i'\in\phi^{-1}(i)} \vY_{i',k}^{\gamma_i} \vY_{i',k+1}^{\gamma_i} \prod_{j\neq i}\prod_{j' \in \phi^{-1}(j)} \vY_{j',k+\virtual{c}_{j'i'}}^{\gamma_i \virtual{C}_{j'i'}}.
\end{gather*}
On the other hand,
\begin{gather*}
\prod_{i'\in \phi^{-1}(i)} \vA_{i',k}^{\gamma_i} = \prod_{i'\in \phi^{-1}(i)} \vY_{i',k}^{\gamma_i} \vY_{i',k+1}^{\gamma_i} \prod_{j'\neq i'} \vY_{j',k+\virtual{c}_{j'i'}}^{\gamma_i \virtual{C}_{j'i'}},
\end{gather*}
which agrees with the calculation of $v(A_{i,k})$ by use of equation \eqref{eq:lindep} and the fact that $\virtual{C}_{j'i'} = 0$ if $i',j' \in \phi^{-1}(i)$.
\end{proof}

\begin{thm}\label{thm:monomial_virtualization} The map $v$ from Lemma~{\rm \ref{lem:AtoAhat}} is a virtualization map.
\end{thm}

\begin{proof}We show the case for $\MM(\infty)_{\cc}$ as the case for $\MM_{\cc}$ is similar. By equation \eqref{eq:Minfge0}, write
\begin{gather*}
M = \prod_{i\in I} \prod_{k\ge 0} Y_{i,k}^{y_i(k)}\qquad \text{and} \qquad \virtual{M} := v(M) = \prod_{i\in I} \prod_{k\ge 0} \prod_{i' \in \phi^{-1}(i)} \vY_{i',k}^{\gamma_iy_i(k)}.
\end{gather*}
We need to show $\wt(M) = \wt^v(\virtual{M})$, $v(e_iM) = e_i^v \virtual{M}$, and $v(f_iM) = f_i^v \virtual{M}$ for all $i \in I$. First
\begin{gather*}
\widetilde{\phi}\bigl( \wt(M) \bigr) = \widetilde{\phi}\biggl( \sum_{i\in I} \bigg( \sum_{k\ge 0} y_i(k) \bigg) \Lambda_i \biggr)
= \sum_{i\in I} \bigg( \sum_{k\ge 0} y_i(k) \bigg) \sum_{i'\in \phi^{-1}(i)} \gamma_i \vL_{i'} = \virtual{\wt}(\virtual{M}),
\end{gather*}
so $v$ commutes with the weight map.

Now fix $i\in I$ and suppose $k_f \ge0$ is such that $f_iM = MA^{-1}_{i,k_f}$. We need to show that $f^v_i\virtual{M} = \virtual{M}\prod\limits_{i'\in\phi^{-1}(i)}\vA_{i',k_f}^{-\gamma_i}$, but by Lemma~\ref{lem:AtoAhat} it suffices to show, for any $i' \in \phi^{-1}(i)$,
\begin{gather*}
k_f = \min\bigg\{ k\ge 0\colon \virtual{\varphi}_{i'}(\virtual{M}) = \sum_{0\le \ell \le k} \gamma_iy_{i}(\ell) \bigg\} .
\end{gather*}
By definition of $\virtual{\varphi}_{i'}$ and the construction of $\virtual{M}$, for all $i'\in\phi^{-1}(i)$ we have
\begin{gather*}
\virtual{\varphi}_{i'}(\virtual{M}) = \max\bigg\{ \sum_{0\le r \le s} \gamma_i y_i(r)\colon s \ge 0 \bigg\} = \gamma_i\max\bigg\{ \sum_{0\le r \le s} y_i(r) \colon s \ge 0 \bigg\} = \gamma_i \varphi_i(M).
\end{gather*}
Hence $\varphi_i^v(M) = \virtual{\varphi}_{i'}(\virtual{M})/\gamma_i = \varphi_i(M)$ for all $i\in I$.  Now
\begin{align*}
\min\bigg\{ k\ge 0 \colon \virtual{\varphi}_{i'}(\virtual{M}) = \sum_{0\le \ell \le k} \gamma_iy_{i}(\ell) \bigg\} &= \min\bigg\{ k\ge 0 \colon \varphi_i^v(M) = \sum_{0\le \ell \le k} y_{i}(\ell) \bigg\} \\
&= \min\bigg\{ k\ge 0 \colon \varphi_i(M) = \sum_{0\le \ell \le k} y_{i}(\ell) \bigg\} \\
& = k_f,
\end{align*}
as required.
The case showing $e_i^v\virtual{M} = v(e_iM)$ is similar.
\end{proof}

\begin{ex}We construct the highest weight $F_4$ crystal $B(\Lambda_1)$ inside of $B(\vL_2)$ of type $E_6$ under the folding defined by
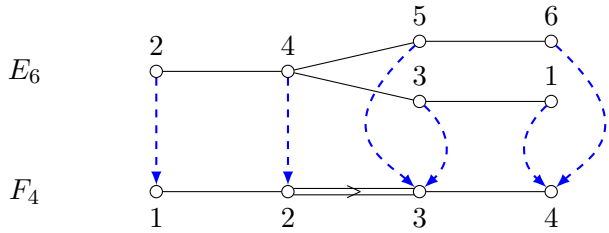
\[
\begin{tikzpicture}[xscale=1.75,yscale=.8]
\node at (0,1) {$E_6$};
\node[dynkdot,label={above:$2$}] (E2) at (1,1) {};
\node[dynkdot,label={above:$4$}] (E4) at (2,1) {};
\node[dynkdot,label={above:$5$}] (E5) at (3,1.5) {};
\node[dynkdot,label={above:$6$}] (E6) at (4,1.5) {};
\node[dynkdot,label={above:$3$}] (E3) at (3,0.5) {};
\node[dynkdot,label={above:$1$}] (E1) at (4,0.5) {};
\path[-]
 (E2) edge (E4)
 (E4) edge (E5)
 (E4) edge (E3)
 (E5) edge (E6)
 (E3) edge (E1);

\def\Foffset{-1}
\node at (0,\Foffset) {$F_4$};
\foreach \x in {1,2,3,4}
{\node[dynkdot,label={below:$\x$}] (F\x) at (\x,\Foffset) {};}
\draw[-] (F1.east) -- (F2.west);
\draw[-] (F3) -- (F4);
\draw[-] (F2.30) -- (F3.150);
\draw[-] (F2.330) -- (F3.210);
\draw[-] (2.55,\Foffset) -- (2.45,\Foffset+.1);
\draw[-] (2.55,\Foffset) -- (2.45,\Foffset-.1);

\path[-latex,dashed,color=blue,thick]
 (E2) edge (F1)
 (E4) edge (F2);
\draw[-latex,dashed,color=blue,thick]
 (E1) .. controls (3.75,\Foffset+1) and (3.75,\Foffset+.5) .. (F4);
\draw[-latex,dashed,color=blue,thick]
 (E3) .. controls (3.25,\Foffset+1) and (3.25,\Foffset+.5) .. (F3);
\draw[-latex,dashed,color=blue,thick]
 (E5) .. controls (2.25,\Foffset+1) and (2.75,\Foffset+.5) .. (F3);
\draw[-latex,dashed,color=blue,thick]
 (E6) .. controls (4.75,\Foffset+1) and (4.25,\Foffset+.5) .. (F4);

\end{tikzpicture}
\]
Using the following code, the virtualization map can be given on a vertex-by-vertex basis.  See Table~\ref{F4toE6}.
\begin{lstlisting}
sage: LaE = RootSystem(['E',6]).weight_lattice().fundamental_weights()
sage: LaF = RootSystem(['F',4]).weight_lattice().fundamental_weights()
sage: new_c = matrix([[0,1,1,1,1,1],[0,0,1,0,1,1],[0,0,0,1,1,1],
....:                 [0,1,0,0,0,1],[0,0,0,1,0,0],[0,0,0,0,1,0]])
sage: ME = crystals.NakajimaMonomials(LaE[2], c=new_c)
sage: MF = crystals.NakajimaMonomials(LaF[4])
sage: phi = {1: [1,6], 2: [3,5], 3: [4], 4: [2]}
sage: sf = {i: 1 for i in MF.index_set()}
sage: v = MF.crystal_morphism(ME.module_generators, virtualization=phi,
....:                         scaling_factors=sf)
\end{lstlisting}

\begin{table}[t]
\begin{align*}
Y_{4,0}  & \mapsto Y_{2,0} &
Y_{3,1} Y_{4,1}^{-1}  & \mapsto Y_{2,1}^{-1} Y_{4,1} \\
Y_{2,2} Y_{3,2}^{-1}  & \mapsto Y_{3,2} Y_{4,2}^{-1} Y_{5,2} &
Y_{1,3} Y_{2,3}^{-1} Y_{3,2}  & \mapsto Y_{1,3} Y_{3,3}^{-1} Y_{4,2} Y_{5,3}^{-1} Y_{6,3} \\
Y_{1,4}^{-1} Y_{3,2}  & \mapsto Y_{1,4}^{-1} Y_{4,2} Y_{6,4}^{-1} &
Y_{1,3} Y_{3,3}^{-1} Y_{4,2}  & \mapsto Y_{1,3} Y_{2,2} Y_{4,3}^{-1} Y_{6,3} \\
Y_{1,4}^{-1} Y_{2,3} Y_{3,3}^{-1} Y_{4,2}  & \mapsto Y_{1,4}^{-1} Y_{2,2} Y_{3,3} Y_{4,3}^{-1} Y_{5,3} Y_{6,4}^{-1} &
Y_{2,4}^{-1} Y_{3,3} Y_{4,2}  & \mapsto Y_{2,2} Y_{3,4}^{-1} Y_{4,3} Y_{5,4}^{-1} \\
Y_{3,4}^{-1} Y_{4,2} Y_{4,3}  & \mapsto Y_{2,2} Y_{2,3} Y_{4,4}^{-1} &
Y_{4,2} Y_{4,4}^{-1}  & \mapsto Y_{2,2} Y_{2,4}^{-1} \\
Y_{3,3} Y_{4,3}^{-1} Y_{4,4}^{-1}  & \mapsto Y_{2,3}^{-1} Y_{2,4}^{-1} Y_{4,3} &
Y_{2,4} Y_{3,4}^{-1} Y_{4,4}^{-1}  & \mapsto Y_{2,4}^{-1} Y_{3,4} Y_{4,4}^{-1} Y_{5,4} \\
Y_{1,5} Y_{2,5}^{-1} Y_{3,4} Y_{4,4}^{-1}  & \mapsto Y_{1,5} Y_{2,4}^{-1} Y_{3,5}^{-1} Y_{4,4} Y_{5,5}^{-1} Y_{6,5} &
Y_{1,6}^{-1} Y_{3,4} Y_{4,4}^{-1}  & \mapsto Y_{1,6}^{-1} Y_{2,4}^{-1} Y_{4,4} Y_{6,6}^{-1} \\
Y_{1,5} Y_{3,5}^{-1}  & \mapsto Y_{1,5} Y_{4,5}^{-1} Y_{6,5} &
Y_{1,6}^{-1} Y_{2,5} Y_{3,5}^{-1}  & \mapsto Y_{1,6}^{-1} Y_{3,5} Y_{4,5}^{-1} Y_{5,5} Y_{6,6}^{-1} \\
Y_{2,6}^{-1} Y_{3,5}  & \mapsto Y_{3,6}^{-1} Y_{4,5} Y_{5,6}^{-1} &
Y_{3,6}^{-1} Y_{4,5}  & \mapsto Y_{2,5} Y_{4,6}^{-1} \\
Y_{4,6}^{-1}  & \mapsto Y_{2,6}^{-1} &
Y_{1,3} Y_{4,3}^{-1}  & \mapsto Y_{1,3} Y_{2,3}^{-1} Y_{6,3} \\
Y_{1,4}^{-1} Y_{2,3} Y_{4,3}^{-1}  & \mapsto Y_{1,4}^{-1} Y_{2,3}^{-1} Y_{3,3} Y_{5,3} Y_{6,4}^{-1} &
Y_{2,4}^{-1} Y_{3,3}^{2} Y_{4,3}^{-1}  & \mapsto Y_{2,3}^{-1} Y_{3,4}^{-1} Y_{4,3}^{2} Y_{5,4}^{-1} \\
Y_{3,3} Y_{3,4}^{-1}  & \mapsto Y_{4,3} Y_{4,4}^{-1} &
Y_{2,4} Y_{3,4}^{-2} Y_{4,3}  & \mapsto Y_{2,3} Y_{3,4} Y_{4,4}^{-2} Y_{5,4} \\
Y_{1,5} Y_{2,5}^{-1} Y_{4,3}  & \mapsto Y_{1,5} Y_{2,3} Y_{3,5}^{-1} Y_{5,5}^{-1} Y_{6,5} &
Y_{1,6}^{-1} Y_{4,3}  & \mapsto Y_{1,6}^{-1} Y_{2,3} Y_{6,6}^{-1}
\end{align*}
\caption{The explicit virtualization map from $B(\Lambda_1)$ of type $F_4$ to $B(\virtual{\Lambda}_2)$ of type $E_6$.}\label{F4toE6}
\end{table}
\end{ex}

\begin{ex}
We consider the folding of type $D_5^{(1)}$ onto $C_3^{(1)}$ given by
\[
\begin{tikzpicture}[xscale=1.75,yscale=1.25]
\node at (0,0) {$D_5^{(1)}$};
\node[dynkdot,label={left:$0$}] (D0) at (1,.5) {};
\node[dynkdot,label={left:$1$}] (D1) at (1,-.5){};
\node[dynkdot,label={above:$2$}] (D2) at (2,0) {};
\node[dynkdot,label={above:$3$}] (D3) at (3,0) {};
\node[dynkdot,label={right:$5$}] (D5) at (4,.5) {};
\node[dynkdot,label={right:$4$}] (D4) at (4,-.5) {};
\draw[-] (D0) -- (D2);
\draw[-] (D1) -- (D2);
\draw[-] (D2) -- (D3);
\draw[-] (D3) -- (D4);
\draw[-] (D3) -- (D5);

\def\Coffset{-1.5}
\node at (0,\Coffset) {$C_3^{(1)}$};
\foreach \x in {0,1,2,3}
{\node[dynkdot,label={below:$\x$}] (C\x) at (\x+1,\Coffset) {};}
\draw[-] (C0.30) -- (C1.150);
\draw[-] (C0.330) -- (C1.210);
\draw[-] (C1) -- (C2);
\draw[-] (C2.30) -- (C3.150);
\draw[-] (C2.330) -- (C3.210);
\draw[-] (3.45,\Coffset) -- (3.55,\Coffset+.1);
\draw[-] (3.45,\Coffset) -- (3.55,\Coffset-.1);
\draw[-] (1.55,\Coffset) -- (1.45,\Coffset+.1);
\draw[-] (1.55,\Coffset) -- (1.45,\Coffset-.1);

\path[-latex,dashed,color=blue,thick]
 (D1) edge (C0)
 (D2) edge (C1)
 (D3) edge (C2)
 (D4) edge (C3);
\draw[-latex,dashed,color=blue,thick]
 (D0) .. controls (0.25,-.5) and (0.25,-1) .. (C0);
\draw[-latex,dashed,color=blue,thick]
 (D5) .. controls (4.75,-.5) and (4.75,-1) .. (C3);
\end{tikzpicture}
\]
We embed $B(\Lambda_0)$ into $B(\vL_0 + \vL_1)$ and construct the virtualization map $v$ by
\begin{lstlisting}
sage: PC = RootSystem(['C',3,1]).weight_lattice(extended=True)
sage: LaC = PC.fundamental_weights()
sage: PD = RootSystem(['D',5,1]).weight_lattice(extended=True)
sage: LaD = PD.fundamental_weights()
sage: MC = crystals.NakajimaMonomials(LaC[0])
sage: MD = crystals.NakajimaMonomials(LaD[0]+LaD[1])
sage: sf = {i: 1 for i in MC.index_set()}
sage: phi = {0: [0,1], 1: [2], 2: [3], 3: [4,5]}
sage: v = MC.crystal_morphism(MD.module_generators, virtualization=phi,
....:                         scaling_factors=sf)
\end{lstlisting}
The resulting first four levels of $B(\Lambda_0)$ and its image under $v$ are given in Fig.~\ref{fig:CD_example}.
\end{ex}

\begin{figure}[t]
\[
\begin{tikzpicture}[>=latex,line join=bevel,xscale=0.8,every node/.style={scale=0.75},yscale=0.5]
\node (node_7) at (124.5bp,9.5bp) [draw,draw=none] {$Y_{0,1} Y_{1,1}^{-1} Y_{2,0} Y_{2,1}^{-1} Y_{3,0} $};
  \node (node_6) at (165.5bp,82.5bp) [draw,draw=none] {$Y_{0,1} Y_{1,1}^{-2} Y_{2,0}^{2} $};
  \node (node_5) at (117.5bp,300.0bp) [draw,draw=none] {$Y_{0,0} $};
  \node (node_4) at (117.5bp,155.5bp) [draw,draw=none] {$Y_{1,0} Y_{1,1}^{-1} Y_{2,0} $};
  \node (node_3) at (29.5bp,9.5bp) [draw,draw=none] {$Y_{1,0} Y_{2,1} Y_{3,1}^{-1} $};
  \node (node_2) at (210.5bp,9.5bp) [draw,draw=none] {$Y_{0,2}^{-1} Y_{2,0}^{2} $};
  \node (node_1) at (117.5bp,228.5bp) [draw,draw=none] {$Y_{0,1}^{-1} Y_{1,0}^{2} $};
  \node (node_0) at (88.5bp,82.5bp) [draw,draw=none] {$Y_{1,0} Y_{2,1}^{-1} Y_{3,0} $};
  \draw [green,->] (node_0) ..controls (71.888bp,61.509bp) and (55.334bp,41.589bp)  .. (node_3);
  \definecolor{strokecol}{rgb}{0.0,0.0,0.0};
  \pgfsetstrokecolor{strokecol}
  \draw (75.0bp,46.0bp) node {$3$};
  \draw [blue,->] (node_1) ..controls (117.5bp,208.04bp) and (117.5bp,189.45bp)  .. (node_4);
  \draw (126.0bp,192.0bp) node {$1$};
  \draw [black,->] (node_6) ..controls (178.04bp,61.722bp) and (190.32bp,42.337bp)  .. (node_2);
  \draw (202.0bp,46.0bp) node {$0$};
  \draw [red,->] (node_6) ..controls (154.08bp,61.722bp) and (142.88bp,42.337bp)  .. (node_7);
  \draw (159.0bp,46.0bp) node {$2$};
  \draw [blue,->] (node_4) ..controls (130.94bp,134.62bp) and (144.23bp,114.96bp)  .. (node_6);
  \draw (156.0bp,119.0bp) node {$1$};
  \draw [black,->] (node_5) ..controls (117.5bp,281.64bp) and (117.5bp,262.65bp)  .. (node_1);
  \draw (126.0bp,265.0bp) node {$0$};
  \draw [blue,->] (node_0) ..controls (95.813bp,63.745bp) and (102.13bp,49.153bp)  .. (108.5bp,37.0bp) .. controls (110.05bp,34.049bp) and (111.78bp,30.973bp)  .. (node_7);
  \draw (117.0bp,46.0bp) node {$1$};
  \draw [red,->] (node_4) ..controls (109.47bp,134.83bp) and (101.66bp,115.71bp)  .. (node_0);
  \draw (115.0bp,119.0bp) node {$2$};
\end{tikzpicture}
\qquad
\begin{tikzpicture}[>=latex,line join=bevel,xscale=0.69,every node/.style={scale=0.75},yscale=0.5]
\node (node_7) at (134.5bp,300.0bp) [draw,draw=none] {$Y_{0,0} Y_{1,0} $};
  \node (node_6) at (134.5bp,155.5bp) [draw,draw=none] {$Y_{2,0} Y_{2,1}^{-1} Y_{3,0} $};
  \node (node_5) at (262.5bp,9.5bp) [draw,draw=none] {$Y_{2,0} Y_{3,1} Y_{4,1}^{-1} Y_{5,1}^{-1} $};
  \node (node_4) at (134.5bp,228.5bp) [draw,draw=none] {$Y_{0,1}^{-1} Y_{1,1}^{-1} Y_{2,0}^{2} $};
  \node (node_3) at (194.5bp,82.5bp) [draw,draw=none] {$Y_{2,0} Y_{3,1}^{-1} Y_{4,0} Y_{5,0} $};
  \node (node_2) at (142.5bp,9.5bp) [draw,draw=none] {$Y_{0,1} Y_{1,1} Y_{2,1}^{-1} Y_{3,0} Y_{3,1}^{-1} Y_{4,0} Y_{5,0} $};
  \node (node_1) at (30.5bp,9.5bp) [draw,draw=none] {$Y_{0,2}^{-1} Y_{1,2}^{-1} Y_{3,0}^{2} $};
  \node (node_0) at (101.5bp,82.5bp) [draw,draw=none] {$Y_{0,1} Y_{1,1} Y_{2,1}^{-2} Y_{3,0}^{2} $};
  \draw [red,->] (node_0) ..controls (112.92bp,61.722bp) and (124.12bp,42.337bp)  .. (node_2);
  \definecolor{strokecol}{rgb}{0.0,0.0,0.0};
  \pgfsetstrokecolor{strokecol}
  \draw (136.0bp,46.0bp) node {$2$};
  \draw [blue,->] (node_6) ..controls (125.36bp,134.83bp) and (116.47bp,115.71bp)  .. (node_0);
  \draw (131.0bp,119.0bp) node {$1$};
  \draw [blue,->] (node_3) ..controls (179.94bp,61.616bp) and (165.54bp,41.964bp)  .. (node_2);
  \draw (184.0bp,46.0bp) node {$1$};
  \draw [black,->] (node_7) ..controls (134.5bp,281.64bp) and (134.5bp,262.65bp)  .. (node_4);
  \draw (143.0bp,265.0bp) node {$0$};
  \draw [blue,->] (node_4) ..controls (134.5bp,208.04bp) and (134.5bp,189.45bp)  .. (node_6);
  \draw (143.0bp,192.0bp) node {$1$};
  \draw [green,->] (node_3) ..controls (213.85bp,61.296bp) and (233.45bp,40.834bp)  .. (node_5);
  \draw (246.0bp,46.0bp) node {$3$};
  \draw [black,->] (node_0) ..controls (81.296bp,61.296bp) and (60.834bp,40.834bp)  .. (node_1);
  \draw (83.0bp,46.0bp) node {$0$};
  \draw [red,->] (node_6) ..controls (151.39bp,134.51bp) and (168.23bp,114.59bp)  .. (node_3);
  \draw (181.0bp,119.0bp) node {$2$};
\end{tikzpicture}
\]
\caption{The first four levels of $B(\Lambda_0)$ of type $C_3^{(1)}$ (left) and the image under the virtualization map with $D_5^{(1)}$ (right).}
\label{fig:CD_example}
\end{figure}
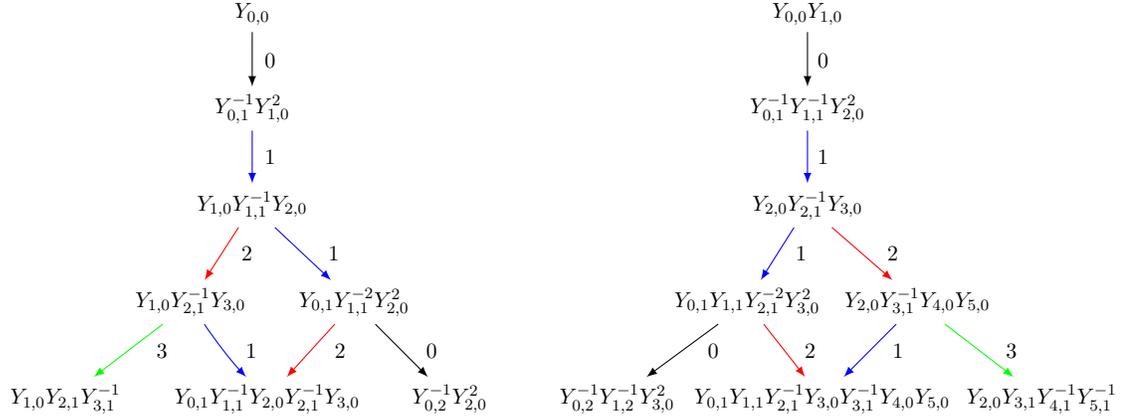

\begin{Remark}
Consider an array $\cc$ for $\MM(\infty)_{\cc}$ such that there exists a total order $\prec$ on $I$ such that $i \prec j$ whenever $c_{ij} = 1$ and $i \sim j$. Let $I_{\prec} := \{i_1 \prec i_2 \prec \cdots \prec i_n\}$. From~\cite{K93}, there exists a~(strict) embedding
\begin{gather}\label{eq:I_path}
B(\infty) \lhook\joinrel\longrightarrow \cdots \otimes \BB_{i_2} \otimes \BB_{i_1} \otimes \BB_{i_n} \otimes \dotsm \otimes \BB_{i_2} \otimes \BB_{i_1},
\end{gather}
where $\BB_i = \{b_i(a) \colon a \in \ZZ\}$ is the elementary $i$-crystal. From the tensor product rule, the image of the embedding~\eqref{eq:I_path} is isomorphic to $\MM(\infty)_{\cc}$ under the map
\begin{gather*}
\prod_{\substack{ i \in I \\ k \in \ZZ_{\geq 0} }}\!\! A_{i,k}^{a_i(k)}\! \mapsto \cdots \otimes b_{i_2}\bigl(a_{i_2}(1)\bigr) \otimes b_{i_1}\bigl(a_{i_1}(1)\bigr) \otimes b_{i_n}\bigl(a_{i_n}(0)\bigr) \otimes \dotsm \otimes b_{i_2}\bigl(a_{i_2}(0)\bigr) \otimes b_{i_1}\bigl(a_{i_1}(0)\bigr),
\end{gather*}
which can be considered a special case of the isomorphism to prove~\cite[Theorem~3.1]{KKS07}. In other words, the values $k_i^e$ and $k_i^f$ determine which tensor factor to act upon. Furthermore, the proof of~\cite[Theorem~5.1]{K96} is a special case of the proof presented here.
\end{Remark}

\section{Virtualization of Kirillov--Reshetikhin crystals}\label{sec:KR_crystals}

We first recall the definition of Nakajima monomials given by Nakajima~\cite{Nakajima01,Nakajima03II,Nakajima03,Nakajima04}. This definition of Nakajima monomials is necessary in order to invoke the results of \cite{HN06}, and this modification is only necessary for this section.  In this case we have $c_{ij} = 1$ for all $i,j \in I$ and instead use
\begin{gather*}
A'_{i,k} = Y_{i,k} Y_{i,k+2} \prod_{j\neq i} Y_{j,k+c_{ji}}^{C_{ji}}
\end{gather*}
in place of $A_{i,k}$ in~\eqref{eq:mon_cry_ops}. In particular, Nakajima showed that if there are no odd length cycles in the Dynkin diagram, then we have an abstract $U_q(\g)$-crystal structure and the analog of Theorem~\ref{thm:highest_weight_model}.\footnote{For type $A_n^{(1)}$, a common generalization of these two models was made in~\cite{ST14}.}
We note that this restricted set of Cartan types includes all finite and affine types except $A_{2k}^{(1)}$.

We note that by using Nakajima's crystal structure, the compatibility condition $\virtual{\cc}$ and $\cc$ is immediately satisfied. Furthermore, it is straightforward to verify that Lemma~\ref{lem:AtoAhat} and Theorem~\ref{thm:monomial_virtualization} hold in this setting.

Now, let us consider $\g$ of affine type and $\lambda$ be a \defn{level-zero dominant weight}; that is to say, we have $\lambda = \sum\limits_{i=1}^N m_i (\Lambda_{r_i} - c_{r_i} \Lambda_0)$ for $m_i \in \ZZ_{\geq 0}$, where the positive integers $c_{r_i}$ are such that $\Lambda_{r_i} - c_{r_i}\Lambda_0$ has level zero \cite[Chapter~12]{kac90}. We consider an extremal level-zero crystal $B(\lambda)$ coming from an extremal level-zero module (see \cite[Section~3]{Kashiwara02} for more information). Kashiwara showed that the crystal exists in~\cite{K94}, and there exists an automorphism $\eta$ of $B(\lambda)$ such that
\begin{gather*}
B(\lambda) / \eta \iso \bigotimes_{i=1}^N(B^{r_i, 1})^{\otimes m_i},
\end{gather*}
where $B^{r,s}$ are certain finite crystals, in~\cite{Kashiwara02}. These finite crystals were later realized to be \defn{Kirillov--Reshetikhin $($KR$)$ crystals}, the crystal basis corresponding to Kirillov--Reshetikhin modu\-les~\cite{HKOTT02, HKOTY99, OS08}.

In~\cite{HN06}, a construction of the extremal level-zero crystal $B(\lambda)$ was given in terms of Nakajima monomials as the closure of certain monomials and the automorphism $\eta$. In particular, a power of the automorphism $\eta$ is a shift map $\tau_s(Y_{i,k}) = Y_{i,k+s}$. It is straightforward to see that the defining monomials, the power of the automorphism, and the shift all agree. Therefore, we obtain the analog results to~\cite[Theorem~4.2]{PS15} and partially~\cite{NS01, NS05II, NS10} of a virtualization of KR crystals by using Nakajima monomials.

\begin{thm}
Let $\g$ be of affine type. Then $v \circ \eta = \eta \circ v$.
Moreover, the KR crystal $B^{r,1}$ virtualizes in $\bigotimes_{r' \in \phi^{-1}(r)} (B^{r',1})^{\otimes \gamma_r}$.
\end{thm}

Thus we have a proof of special cases of~\cite[Conjecture~3.7]{OSS03III} using Nakajima monomials.

\section[Monomials to PBW data in type $A_n$]{Monomials to PBW data in type $\boldsymbol{A_n}$}\label{sec:monomials_pbw}

Let $W$ be the Weyl group associated to the Cartan matrix $C = (C_{ij})$.  In this section, we are considering the Cartan matrix of type $A_n$, so $W$ is isomorphic to the symmetric group $\mathfrak{S}_{n+1}$.  In particular, $W$ is a finite group generated by simple transpositions $s_1, \dotsc, s_n$.  Every element $w \in W$ may be expressed in the form $w = s_{i_1}s_{i_2} \dotsm s_{i_\ell}$; in this case, if $\ell$ is minimal among all such expressions for $w$, then $\ell$ is called the length of $w$ and the expression above is called reduced.  An element $w\in W$ may have more than one reduced expression, but all reduced expressions have the same length.

There is a unique element of longest length in $W$, which we denote as $w_0$.  Associated to a reduced expression $w_0 = s_{i_1}s_{i_2} \dotsm s_{i_N}$, where $N$ is the number of positive roots, one may construct a realization of the crystal $B(\infty)$ using Lusztig's PBW-type basis~\cite{L90}.  This construction can be complicated in large rank.  However, in~\cite{SST17}, an efficient algorithm for calculating the crystal structure for certain reduced expressions was developed.  We now recall this efficient procedure for a specifically chosen reduced expression for $w_0$.

Fix, once and for all, the reduced expression
\begin{gather}
\label{eq:dual_BZL_word}
w_0 = s_n(s_{n-1}s_n) (s_{n-2} s_{n-1} s_n) \dotsm (s_1 s_2 \dotsm s_n).
\end{gather}
The reduced expression in equation~\eqref{eq:dual_BZL_word} is called the \defn{dual BZL word}.

\begin{Lemma}[\cite{SST17}]\label{lemma:simply_braided}
The word $\ii = (n,n-1,n,\dots,1,2,\dots,n)$ is simply braided.
\end{Lemma}

We will not need the explicit notion of a word being simply braided.  For us, this means that the crystal structure on $B(\infty)$ can be understood efficiently using a bracketing rule on Kostant partitions.\footnote{This is the same crystal structure as defined in, for example,~\cite{BFZ96}, but the combinatorics defining the rules is different.}  Let $\Phi^+ = \{\alpha_{j,k} = \alpha_j + \alpha_{j+1} + \cdots + \alpha_k \colon 1\le j \le k \le n\}$ denote the set of positive roots of type $A_n$, define $\mathcal{R} = \{ (\alpha) \colon \alpha \in \Phi^+\}$, and set $\Kp(\infty)$ to be the free $\ZZ_{\ge0}$-span of $\R$.  An element of $\Kp(\infty)$ will be considered a Kostant partition, and denoted by $\bm\alpha = \sum\limits_{(\alpha)\in\mathcal{R}} c_\alpha(\alpha)$, where $c_\alpha \in \ZZ_{\ge0}$.  For some fixed $i\in I$, in order to compute $f_i\bm\alpha$ and $e_i\bm\alpha$, one must first calculate the bracketing sequence associated to $i$ and $\bm\alpha$.  Indeed, following~\cite{SST17}, define $S_i(\bm\alpha)$ to be the string of brackets
\begin{gather*}
\underbrace{)\cdots)}_{c_{\alpha_{i,n}}}\
\underbrace{(\cdots(}_{c_{\alpha_{i+1,n}}}\
\underbrace{)\cdots)}_{c_{\alpha_{i,n-1}}}\
\underbrace{(\cdots(}_{c_{\alpha_{i+1,n-1}}}\ \ \
\cdots\ \ \
\underbrace{)\cdots)}_{c_{\alpha_{i,i+1}}}\
\underbrace{(\cdots(}_{c_{\alpha_{i+1,i+1}}}\
\underbrace{)\cdots)}_{c_{\alpha_{i,i}}}.
\end{gather*}
Successively cancel $()$-pairs to obtain sequence of the form $)\cdots)(\cdots($. We call the remaining brackets \defn{uncanceled}.

\begin{dfn} \label{def:KPops}
Let $i \in I$ and $\bm\alpha =  \sum\limits_{(\alpha)\in \R} c_\alpha(\alpha)\in\Kp(\infty)$.
\begin{itemize}\itemsep=0pt
\item Let $\beta$ be the root corresponding to the rightmost uncanceled `$)$' in $S_i(\bm\alpha)$.  Define
\begin{gather*}
e_i\bm\alpha = \bm\alpha - (\beta) + (\beta-\alpha_i).
\end{gather*}
If $\beta=\alpha_i$, we interpret $(0)$ as the additive identity in $\Kp(\infty)$.  If no such `$)$' exists, then~$e_i\bm\alpha$ is undefined.
\item Let $\gamma$ denote the root corresponding to the leftmost uncanceled `$($' in $S_i(\bm\alpha)$.  Define,
\begin{gather*}
f_i\bm\alpha = \bm\alpha - (\gamma) + (\gamma+\alpha_i).
\end{gather*}
If no such `$($' exists, set $f_i\bm\alpha = \bm\alpha + (\alpha_i)$.
\item $\wt(\bm\alpha) = -\sum\limits_{\alpha\in\Phi^+} c_\alpha\alpha.$
\item $\varepsilon_i(\bm\alpha) = \text{number of uncanceled `$)$' in the bracketing sequence of $\bm\alpha$}$.
\item $\varphi_i(\bm\alpha) = \varepsilon_i(\bm\alpha) + \langle h_i , \wt(\bm\alpha) \rangle$.
\end{itemize}
\end{dfn}

\begin{ex}Let $n=3$ and define $\bm\alpha = 2(\alpha_3) + 3(\alpha_2+\alpha_3) + 2(\alpha_2) + 2(\alpha_1+\alpha_2) + 5(\alpha_1)$.  Then
\begin{gather*}
S_1(\bm\alpha)  = {(}\ (\ (\ )\ )\ (\ (\ )\ )\ )\ )\ )\, , \qquad
S_2(\bm\alpha)  = {)}\ )\ )\ (\ (\ )\ )\, ,\qquad
S_3(\bm\alpha)  = {)}\ )\, .
\end{gather*}
Hence
\begin{gather*}
f_1\bm\alpha = 2(\alpha_3) + 3(\alpha_2+\alpha_3) + 2(\alpha_2) + 2(\alpha_1+\alpha_2) + 6(\alpha_1),\\
f_2\bm\alpha = 2(\alpha_3) + 3(\alpha_2+\alpha_3) + 3(\alpha_2) + 2(\alpha_1+\alpha_2) + 5(\alpha_1),\\
f_3\bm\alpha = 3(\alpha_3) + 3(\alpha_2+\alpha_3) + 2(\alpha_2) + 2(\alpha_1+\alpha_2) + 5(\alpha_1).
\end{gather*}
\end{ex}

\begin{thm}[\cite{SST17}]Let $u_\infty$ be the unique element of $B(\infty)$ with weight zero.  The map \smash{$u_\infty \mapsto (0)$} induces a $U_q(A_n)$-crystal isomorphism between~$B(\infty)$ and $\Kp(\infty)$.
\end{thm}

For the remainder of this section, define $\cc = (c_{ij})_{i\neq j}$ by $c_{ij} = 0$ if $i > j$ and $c_{ij} = 1$ if $i < j$.  This is the same convention used in~\cite{KKS07,KS08}.

\begin{thm}\label{thm:monomial_Kostant}The map defined by
\begin{gather*}
\prod_{1\le j \le k \le n} \left( \prod_{p=j}^k A_{p,k-p}^{-1} \right)^{\ell_{j,k}} \mapsto \sum_{j=1}^n\sum_{k=j}^n \ell_{j,k}(\alpha_{j,k})
\end{gather*}
is a $U_q(A_n)$-crystal isomorphism from $\MM(\infty)_\cc$ to Kostant partitions.
\end{thm}

Before proving the theorem, it must be shown that such a map is well-defined.

\begin{Lemma}\label{lemma:monomials_type_A}Each $M \in \MM(\infty)_\cc$ can be written uniquely in the form
\begin{gather*}
M = \prod_{1\le j \le k \le n} \left( \prod_{p=j}^k A_{p,k-p}^{-1} \right)^{\ell_{j,k}},
\end{gather*}
for some $\ell_{j,k} \in \ZZ_{\ge0}$.
\end{Lemma}

\begin{proof}In~\cite[Theorem~4.1]{KS08}, it was shown that
\begin{gather*}
\prod_{i \in I} \prod_{q=0}^{n-i} A_{i,q}^{-a_{i,q}}  \in \MM(\infty)_\cc
\end{gather*}
if and only if
\begin{gather}\label{eq:monomial_condition}
0 \leq a_{1,i-1} \leq a_{2,i-2} \leq \cdots \leq a_{i,0}
\end{gather}
for all $1 \leq i \leq n$. It is straightforward to see that we have
\begin{gather*}
a_{i,q} = \ell_{i,i+q} + \ell_{i-1,i+q} + \cdots + \ell_{1,i+q},\\ \ell_{i,i+q} = a_{i,q} - a_{i-1,q+1},
\end{gather*}
and that $\ell_{i,k} \in \ZZ_{\geq 0}$ if and only if~\eqref{eq:monomial_condition} holds. Hence, the claim follows.
\end{proof}

\begin{proof}[Proof of Theorem~\ref{thm:monomial_Kostant}] Let $i\in I= \{1,2,\dots,n\}$.  For notational brevity, once and for all in this proof, denote
\begin{alignat*}{3}
& M= \prod_{1\le j \le k \le n} \left( \prod_{p=j}^k A_{p,k-p}^{-1} \right)^{\ell_{j,k}}, \qquad & & f_iM = \prod_{1\le j \le k \le n} \left( \prod_{p=j}^k A_{p,k-p}^{-1} \right)^{\ell_{j,k}'},& \\
& \bm\alpha= \sum_{1\le j \le k \le n} \ell_{j,k} (\alpha_{j,k}),\qquad & & f_i\bm\alpha= \sum_{1\le j \le k \le n} \ell_{j,k}'' (\alpha_{j,k}).&
\end{alignat*}
Recall
\begin{gather*}
S_i(\bm\alpha) = \underbrace{)\cdots)}_{\ell_{{i,n}}}\
\underbrace{(\cdots(}_{\ell_{{i+1,n}}}\
\underbrace{)\cdots)}_{\ell_{{i,n-1}}}\
\underbrace{(\cdots(}_{\ell_{{i+1,n-1}}}\ \ \
\cdots\ \ \
\underbrace{)\cdots)}_{\ell_{{i,i+1}}}\
\underbrace{(\cdots(}_{\ell_{{i+1,i+1}}}\
\underbrace{)\cdots)}_{\ell_{{i,i}}}.
\end{gather*}
Suppose the leftmost uncanceled `(' in $S_i(\bm\alpha)$ occurs in the position corresponding to the root $\alpha_{i+1,k_0}$, for some $i < k_0 \le n$.  Then $f_i\bm\alpha = \bm\alpha - (\alpha_{i+1,k_0}) + (\alpha_{i,k_0})$. Hence
\begin{gather*}
\ell_{j,k}'' = \begin{cases}
\ell_{i,k_0} + 1 & \text{if } (j,k) = (i,k_0) , \\
\ell_{i,k_0} - 1 & \text{if } (j,k) = (i+1,k_0), \\
\ell_{j,k} & \text{otherwise}.
\end{cases}
\end{gather*}
If there is no uncanceled `(' in $S_i(\bm\alpha)$, then set $k_0 = i$; therefore, $f_i\bm\alpha = \bm\alpha + (\alpha_{i,i})$ and
\begin{gather*}
\ell_{j,k}'' = \begin{cases}
 \ell_{i,i} + 1 & \text{if } (j,k) = (i,i) , \\
 \ell_{j,k} & \text{otherwise}.
\end{cases}
\end{gather*}
We intend to show $k_i^f(M) = k_0 - i$; in this case, a direct calculation shows that $\ell_{j,k}' = \ell_{j,k}''$ for all $1 \le j \le k \le n$.

Expanding the given expression for $M$, we have
\begin{align*}
M & = \prod_{1\le j \le k \le n} \left( \prod_{p=j}^k A_{p,k-p}^{-1} \right)^{\ell_{j,k}}
\\ & = \prod_{1\le j \le k \le n} \left( \prod_{p=j}^k Y_{p,k-p}^{-1} Y_{p,k-p+1}^{-1} Y_{p-1,k-p+1} Y_{p+1,k-p} \right)^{\ell_{j,k}}
\\ & = \prod_{1\le j \le k \le n} \big( Y_{j-1,k-j+1} Y_{j,k-j+1}^{-1} Y_{k,0}^{-1} Y_{k+1,0} \big)^{\ell_{j,k}}.
\end{align*}
If we express this same $M$ using the form $\prod\limits_{p\in I} \prod\limits_{q\ge 0} Y_{p,q}^{y_p(q)}$, then, for $i\in I$,
\begin{gather*}
\sum_{r=0}^s y_i(r) = -\ell_{i,i} + \sum_{t=1}^i (\ell_{t,i-1} - \ell_{t,i}) + \sum_{u=0}^{s-1} (\ell_{i+1,i+1+u} - \ell_{i,i+u}).
\end{gather*}
If there is no uncanceled `(' in $S_i(\bm\alpha)$, then each term in the third summand is nonpositive.  Therefore
\begin{gather*}
\varphi_i(M) = \max\left\{ \sum_{r=0}^s y_i(r) \colon s \ge 0 \right\} = 0
\end{gather*}
and
\begin{gather*}
k_i^f(M) = \min\left\{ s \ge 0 \colon \varphi_i(M) = \sum_{r=0}^s y_i(r) \right\} = 0 = k_0 -i.
\end{gather*}
On the other hand, if there is an uncanceled `(' in $S_i(\bm\alpha)$ (which corresponds to $(\alpha_{i+1,k_0})$ by our definitions), then $\ell_{i+1,i+1+t} \le \ell_{i,i+t}$ for $t > k_0-i$.  In particular, the maximum
\begin{gather*}
\max\left\{ \sum_{r=0}^s y_i(r) \colon s \ge 0 \right\}
\end{gather*}
is attained when $s = k_0-i$.  This maximum cannot be attained for any $s_0 < s$ since, if it did, it would contradict the position of the leftmost uncanceled `(' in $S_i(\bm\alpha)$.  Hence $k_i^f(M) = k_0-i$, as required.
\end{proof}

An alternative approach to proving Theorem~\ref{thm:monomial_Kostant} would be to extend the map given in~\cite[Proposition~6.12]{Lee14} to a map from $\MM(\infty)$ to the (marginally) large tableaux model~\cite{Cliff98,HL08}. From there, one could construct the Kostant partition using the BZL (or minimal lexicographic) word
\begin{gather*}
w_0 = s_1(s_2 s_1) (s_3 s_2 s_1) \dotsm (s_n s_{n-1} \dotsm s_1)
\end{gather*}
and then change that into the dual BZL word to recover Theorem~\ref{thm:monomial_Kostant}.

\section{Open problems}\label{sec:problems}

An original goal of this paper was to use Theorems~\ref{thm:monomial_virtualization} and~\ref{thm:monomial_Kostant} to give a method to extract Lusztig data from Nakajima monomials. However, because of the compatibility condition of $\cc$ and the specific reduced word (and hence~$\cc$ that comes from the corresponding orientation of the Dynkin diagram) needed for Theorem~\ref{thm:monomial_Kostant}, our techniques do not apply. In this section, we discuss these limitations in more detail.

We first note that equation~\eqref{eq:weight_embedding} gives a virtualization of Lusztig data in types $B_n$ and $C_n$ in terms of type $A_{2n-1}$, and this corresponds to the virtualization map given in~\cite{JS15, NS08III} (in terms of MV polytopes). Indeed, let~$\ii$ denote a reduced expression for $w_0 \in W$ corresponding to a~finite type~$\g$, and fix a diagram automorphism $\phi$.  Then
\begin{gather*}
\virtual{\ii} := \overrightarrow{\prod_{i \in \ii}} \prod_{j \in \phi^{-1}(i)} j
\end{gather*}
is a reduced expression for $\virtual{w}_0 \in \virtual{W}$. Let $\mathcal{L} = (L_i)_{i \in \ii}$ denote the Lusztig data for an element in the set of MV polytopes of type $\g$. Then we define
\begin{gather}\label{eq:virtualization_Luzstig_data}
v(\mathcal{L}) = (\gamma_{\phi(j)} L_{\phi(j)})_{j \in \phi^{-1}(\ii)}.
\end{gather}
However, there is no reduced word $\ii$ for $w_0$ in type $B_n$ (equiv.~$C_n$) such that $\virtual{\ii}$ is the reduced word given by Lemma~\ref{lemma:simply_braided}, which is used in Theorem~\ref{thm:monomial_Kostant}. Thus, one way to solve this problem would be determine an explicit isomorphism
\begin{gather*}%\label{eq:mutation_map}
\psi_{\virtual{\cc},\virtual{\cc}'} \colon \  \virtual{\MM}(\infty)_{\virtual{\cc}} \longrightarrow \virtual{\MM}(\infty)_{\virtual{\cc}'}.
\end{gather*}

We recall a method to mutate $\MM_{\cc}$ into another with a different array $\MM_{\cc'}$ from~\cite{K96}. Fix some $\mm = (m_i \in \ZZ \colon i \in I)$, which we will call a \defn{mutation}, and define the array $\cc' = \{c'_{ij} \colon i,j \in I$, $i \neq j\}$ by $c'_{ij} = c_{ij} + m_i - m_j$. This induces a crystal isomorphism $\mu_{\mm} \colon \MM_{\cc} \longrightarrow \MM_{\cc'}$ given by $Y_{i,k} \mapsto Y_{i,k+m_i}$. Therefore, we have an explicit isomorphism
\begin{gather}
\label{eq:mutation_hw}
\mu_{\mm} \colon \ \MM(M)_{\cc} \longrightarrow \MM\bigl(\mu_{\mm}(M)\bigr)_{\cc'}
\end{gather}
for any fixed monomial $M$ by the restriction of $\mu_{\mm}$ to $\MM(M)$.

First, note that there exists $\virtual{\mm}$ such that $v \circ \mu_{\mm} = \mu_{\virtual{\mm}} \circ v$ by considering $\virtual{m}_j  = m_i$ for all $j \in \phi^{-1}(i)$. Next, from~\cite{K96}, we can always find a mutation $\mm$ such that $\virtual{\cc}$ is compatible with $\cc$ when the Dynkin diagram has no loops. Indeed, if we consider a linear Dynkin diagram oriented towards an endpoint $v_h$, we can flip the orientation of a single edge $v_1 \to v_0$ by considering the mutation sequence $m_{v_1} = \cdots = m_{v_t} = 1$, where $v_t$ is the other endpoint of the Dynkin diagram, and all other $m_i = 0$. We can iterate this (which corresponds to adding the mutations considered as vectors) for edges successively closer to $v_t$, thus we can obtain the Dynkin diagram oriented with both endpoints being sinks. Note that there may be values $c_{ij} \notin \ZZ_{\geq 0}$, but these will all correspond to $i \not\sim j$.

\begin{ex}Consider a Dynkin diagram that is an edge decorated path of length $3$; e.g., type $B_4$ or $C_3^{(1)}$. Then we have
\begin{gather*}
\cc :=
\begin{bmatrix}
\bullet & 1 & 1 & 1 \\
0 & \bullet & 1 & 1 \\
0 & 0 & \bullet & 1 \\
0 & 0 & 0 & \bullet
\end{bmatrix}
\xrightarrow[\hspace{60pt}]{\mu_{(0,0,1,2)}}
\begin{bmatrix}
\bullet & 1 & 0 & -1 \\
0 & \bullet & 0 & -1 \\
1 & 1 & \bullet & 0 \\
2 & 2 & 1 & \bullet
\end{bmatrix} =: \cc'.
\end{gather*}
We note that the only entries $c_{ij}$ with $i \sim j$ correspond to those on the superdiagonal and subdiagonal. We can compute this example with the following code.
\begin{lstlisting}
sage: def mutate(c, m):
....:     cp = copy(c)
....:     for i in range(c.nrows()):
....:         for j in range(c.ncols()):
....:             cp[i,j] += m[i] - m[j]
....:     return cp
sage: c = matrix([[0,1,1,1],[0,0,1,1],[0,0,0,1],[0,0,0,0]])
sage: mutate(c, [0,0,1,2])
[ 0  1  0 -1]
[ 0  0  0 -1]
[ 1  1  0  0]
[ 2  2  1  0]
\end{lstlisting}
\end{ex}

Recall that $\MM(\infty)_\cc$ can be defined as $\varinjlim\limits_{\lambda \to \infty} Y_{\lambda}^{-1} \MM(Y_{\lambda})$. However, trying to take the corresponding virtualization of highest weight crystals means we have to do a renormalization. From~\cite{K96}, there exists an abstract isomorphism of crystals $\MM(M)_{\cc} \cong \MM(M')_{\cc'}$, where~$M$ and~$M'$ are highest weight monomials with $\wt(M) = \wt(M')$. However, no combinatorial description of this isomorphism is known beyond equation~\eqref{eq:mutation_hw}. Therefore, we cannot renormalize to have the highest weight monomial be~$Y_{\lambda}$ for any choice of $\virtual{\cc}'$. Instead, we currently must describe $\psi_{\virtual{\cc}, \virtual{\cc}'}$ using a sequence of transition maps on the polyhedral model from~\cite{BZ01} to connect the Nakajima monomials with the Lusztig data in types~$B_n$ or~$C_n$ using the virtualization map.

To give a little more detail, fix a reduced word $\ii$ for $w_0$ of type $B_n$ (or $C_n$) and array $\cc$. Let~$\virtual{\cc}'$ and~$\virtual{\ii}'$ be the array and reduced word for Theorem~\ref{thm:monomial_Kostant}. Consider the commutative diagram:
\[
\begin{tikzpicture}[xscale=3,yscale=1.5]
\node (Mc)   at (0,1) {$\MM(\infty)_\cc$};
\node (hMc)  at (1,1) {$\virtual{\MM}(\infty)_{\virtual{\cc}}$};
\node (hMc') at (2,1) {$\virtual{\MM}(\infty)_{\virtual{\cc}'}$};
\node (Kp)   at (0,0) {$\Kp(\infty)_{\ii}$};
\node (hKp1) at (1,0) {$\virtual{\Kp}(\infty)_{\virtual{\ii}}$};
\node (hKp2) at (2,0) {$\virtual{\Kp}(\infty)_{\virtual{\ii}'}$,};
\path[->,font=\scriptsize]
 (Mc)   edge node[above]{$v$} (hMc)
 (hMc)  edge node[above]{$\psi_{\virtual{\cc}, \virtual{\cc}'}$} (hMc')
 (Mc)   edge node[left]{$\zeta$} (Kp)
 (hMc)  edge node[left]{$\xi$} (hKp1)
 (hMc') edge node[right]{$\Upsilon$} (hKp2)
 (Kp)   edge node[below]{$v$} (hKp1);
\draw [double equal sign distance]
 (hKp1) to (hKp2);
\end{tikzpicture}
\]
where $\Kp(\infty)_{\ii}$ is the crystal of Kostant partitions with respect to the reduced word $\ii$ and $\Upsilon$ is the isomorphism from Theorem~\ref{thm:monomial_Kostant}. The maps $\xi$ and $\zeta$ are abstract crystal isomorphisms for which we currently do not know an explicit description other than by using transition maps of~\cite{BZ01} and the diagram above. Therefore, determining an explicit map without using a sequence of transitions maps for $\psi_{\virtual{\cc}, \virtual{\cc}'}$ is equivalent to an explicit description of $\xi$. Furthermore, a description of $\psi_{\virtual{\cc}, \virtual{\cc}'}$ gives an explicit map for $\zeta$ by Theorem~\ref{thm:monomial_virtualization} and equation~\eqref{eq:virtualization_Luzstig_data}. Note that having an explicit map for $\zeta$ only partially gives a map for $\psi_{\virtual{\cc},\virtual{\cc}'}$ or $\xi$ (i.e., when restricted to the virtual crystal).

\begin{problem}Determine an analog of Theorem~{\rm \ref{thm:monomial_Kostant}} for other types.
\end{problem}

For type $A_n^{(1)}$, there exists an explicit crystal isomorphism from Nakajima monomials and generalized young walls given by composition of~\cite[Theorem~5.1]{KKS07} and~\cite[Theorem~3.1]{KS10}, which uses the path model of~\cite{KKM94}.  (See also \cite[Remark~2.4]{KS08}.) Kim and Shin gave an extension of generalized Young walls to other types using zigzag strip bundles~\cite{KS13,KS14,KS15,KS16}. Another potential application of our results is to give a virtualization map from zigzag strip bundles to generalized Young walls through the explicit crystal isomorphism with Nakajima monomials. However, this has the same technical limitation as before.

Similarly, we can apply our virtualization map to the KR crystals of~\cite{GS16}, but that requires~$\cc$ to correspond to orienting the Dynkin diagram around the cycle. Thus, this is a similar technical limitation as before, but has a more severe limitation because there is not a way to apply mutations to the cycle to obtain a compatible~$\cc$. Additionally, one needs to construct~$B^{n,s}$ using Nakajima monomials, but this should be similar to~$B^{1,s}$ for a~$\virtual{\cc}$. Yet, it is likely that a~solution to the Nakajima monomial virtualization for KR crystals is equivalent to virtualization for zigzag strip bundles via the Kyoto path model~\cite{KKMMNN91, KKMMNN92, OSS03IV} and~\cite[Theorem~4.2]{GS16}.

\subsection*{Acknowledgements}
B.S.\ was partially supported by CMU Early Career grant \#C62847 and by Simons Foundation grant \#429950. T.S.\ was partially supported by the National Science Foundation RTG grant NSF/DMS-1148634. This work benefited from computations using \sage~\cite{combinat, sage}.

\pdfbookmark[1]{References}{ref}
\LastPageEnding

\end{document}